\documentclass[12pt,reqno]{amsart}
\usepackage[a4paper,margin=30mm]{geometry}
\usepackage[T1]{fontenc}
\usepackage{lmodern,microtype,mathtools,amssymb,mathrsfs,xcolor}
\usepackage[hidelinks]{hyperref}
\numberwithin{equation}{section}
\newtheorem{theorem}{Theorem}[section]
\newtheorem{proposition}[theorem]{Proposition}
\newtheorem{lemma}[theorem]{Lemma}
\newtheorem{corollary}[theorem]{Corollary}
\theoremstyle{remark}
\newtheorem{remark}[theorem]{Remark}
\DeclareMathOperator{\Ric}{Ric}
\DeclareMathOperator{\Rm}{Rm}
\DeclareMathOperator{\Vol}{Vol}
\DeclareMathOperator{\inj}{inj}
\DeclareMathOperator{\conj}{conj}
\DeclareMathOperator{\Spec}{Spec}
\DeclareMathOperator{\Tr}{Tr}
\newcommand{\dV}{\,dV_g}
\newcommand{\1}{\mathbf 1}
\title[Short geodesics on isospectral three-manifolds]{Shortest closed geodesics on three-dimensional isospectral manifolds}
\author[Y. Li, Y. Wu, J. Zhou]{Yuxiang Li, Yunqing Wu, Jie Zhou}
\address{
	The Departmental of Mathematical Sicences\\
	Tsinghua University, Beijing, People's Republic of China}
\email{liyuxiang@tsinghua.edu.cn}
\address{
	The Institute of Geometry and Physics\\
	University of Science and Technology of China, Hefei, Anhui, People's Republic of China}
\email{yqwu19@ustc.edu.cn}
\address{
	School of Mathematical Sciences \\
	Capital Normal University, Beijing, People's Republic of China}
\email{zhoujiemath@cnu.edu.cn}
\date{}
\hypersetup{pdftitle={Shortest closed geodesics on three-dimensional isospectral manifolds},pdfauthor={Yuxiang Li, Yunqing Wu, Jie Zhou}}

\begin{document}
\begin{abstract}
We prove that the spectrum of a smooth closed three-manifold
gives a positive lower bound for the length of its shortest nonconstant closed geodesic. The argument uses
three curvature heat coefficients and the first positive singularity of the retarded wave trace. The heat coefficients yield a curvature--length dichotomy: sufficiently large maximum curvature forces a closed geodesic shorter than the curvature scale.  We then show that a shortest closed geodesic
lying below the  conjugate radius produces a singularity of the
sine trace.  Combining these estimates gives a lower bound depending only
on the common spectrum. The compactness theorem of Anderson then
implies smooth compactness of the full isospectral family modulo
diffeomorphisms.
\end{abstract}
\maketitle

\section{Introduction}

Let $(M,g)$ be a smooth, closed, connected Riemannian manifold.  We use
the nonnegative Laplace--Beltrami operator on functions,
$\Delta_g=-\operatorname{div}_g\nabla$, and write
\[
 0=\lambda_0(g)<\lambda_1(g)\le \lambda_2(g)\le\cdots
\]
for its spectrum, repeated according to multiplicity.  The question of
how much geometry is determined by this sequence goes back at least to
Kac's formulation of the inverse spectral problem \cite{Kac1966}.
Uniqueness fails in considerable generality.  Milnor constructed
isospectral nonisometric flat tori \cite{Milnor1964}; Vign\'eras found
closed hyperbolic examples \cite{Vigneras1980}; and Sunada's method
produced a general covering construction \cite{Sunada1985}.  There are
also continuous families of isospectral metrics, including conformally
equivalent examples \cite{BrooksGordon1990} and examples on simply connected manifolds \cite{Schueth1999}.  Such examples rule out spectral rigidity,
but they do not by themselves rule out compactness of the set of all
metrics with one fixed spectrum.

The compactness problem has a different history in low dimensions.  For
closed surfaces, Wolpert's analysis of spectral degeneration
\cite{Wolpert1987} and the determinant and uniformization method of
Osgood, Phillips and Sarnak \cite{OPS1988} gave the foundational
compactness results.  In dimension three, Brooks, Perry and Yang first
proved compactness in a conformal class containing a metric of negative
constant scalar curvature \cite{BPY1989}; Chang and Yang subsequently
treated general conformal classes \cite{ChangYang1989,ChangYang1990}.
For unrestricted isospectral families, Anderson showed that a uniform
positive lower bound for the length of the shortest closed geodesic
implies smooth compactness modulo diffeomorphisms
\cite{Anderson1991}.  Anderson's compactness theorem can be stated as follows. 
\begin{theorem}[Anderson~\cite{Anderson1991}] 
The space of compact isospectral 3-manifolds $(M,g)$ for which the 
length $\ell_g$ of the shortest closed geodesic is bounded below, 
is compact in the $C^\infty$ topology modulo diffeomorphisms, and contains only finitely many diffeomorphism type.
\end{theorem}

Further conditional compactness theorems were
obtained by Brooks, Perry and Petersen \cite{BPP1992,BPP1994}, Gursky
\cite{Gursky1993}, and Zhou \cite{Zhou1997}.  These results reduce the
unrestricted three-dimensional problem to the possible collapse of the
shortest closed geodesic.

Dimension four exhibits additional concentration phenomena.  Xu proved
compactness results for conformal isospectral classes under curvature
hypotheses \cite{Xu1995a,Xu1995b}, and Chen and Xu developed a more
systematic four-dimensional theory using spectral curvature identities
and conformal analysis \cite{ChenXu1996}.  Later work includes Gursky's
compactness theorem under integral curvature assumptions
\cite{Gursky1993}, the positive-Yamabe and small-Weyl theorem of Liu and
Wang \cite{LiuWang2019}, and Xu's analysis of weak compactness and
concentration for conformal isospectral metrics \cite{KeXu2019}.  These
four-dimensional theorems require conformal or curvature assumptions;
they also indicate why the three-dimensional argument below is tied to
the special form of the low-dimensional heat invariants.

We now restrict to dimension three.  Let $\ell(g)$ denote the shortest
length of a nonconstant smooth closed geodesic.  This includes
contractible closed geodesics and is therefore different from the
homotopy systole.  We write $\inj(g)$ and $\conj(g)$ for the
injectivity and conjugate radii, respectively.  The underlying manifold
is allowed to vary.  Our main result removes the remaining length
hypothesis from Anderson's theorem.

\begin{theorem}\label{thm:main}
Let $\Lambda$ be the spectrum of a smooth closed connected
three-manifold.  There exists $\varepsilon_\Lambda>0$ such that every
smooth closed connected three-manifold $(M,g)$ with
$\Spec(\Delta_g)=\Lambda$ satisfies
\[
                         \ell(g)\ge\varepsilon_\Lambda.
\]
Consequently, the class of all such manifolds is compact in the
$C^\infty$ topology,and contains only finitely many diffeomorphism type.
\end{theorem}

The consequence follows from Anderson's compactness theorem
\cite{Anderson1991}, 
\cite[Theorem~3.5]{Perry2003}; the heat-invariant bootstrap of Brooks,
Perry and Petersen \cite{BPP1994} supplies the corresponding curvature
control.  The new point is the spectral lower bound for $\ell(g)$.

The wave trace is the natural spectral object for this purpose.  The
connection between singularities of wave traces and lengths of closed
geodesics was established by Chazarain \cite{Chazarain1974} and
Duistermaat and Guillemin \cite{DuistermaatGuillemin1975}. 
For the classical wave trace $$\mathcal{W}(t)=\frac{1}{2}\sum_{j\ge 0}e^{\pm i\sqrt{\lambda _j}t},$$ Chazarain\cite{Chazarain1974} proved
\begin{align}\label{eq:motivation}
\mathrm{singsupp}\,\mathcal{W}\subset \mathcal{L}\cup \{0\},
\end{align}
there $\mathcal{L}$ is the set of lengths (and their opposites) of the closed geodesics of $M$. Noting $l(g)=\inf \mathcal{L}$ and  $\mathcal{W}(t)$( and its singular support) only depends on the spectrum.  \eqref{eq:motivation}  strongly indicates $l(g)$ is spectrum invariant. However,  
standard trace formulas are most transparent only when the relevant fixed-point set is clean or the closed orbit is nondegenerate.  A shortest geodesic need satisfy neither hypothesis, and contributions from different return
directions may in principle cancel.  So, the inverse inclusion of \eqref{eq:motivation} may not holds, which means  $\mathcal{W}$ may also be smooth at  $t=l(g)$, and $l(g)$ may not be recognized from the singularity of $\mathcal{W}$.  We avoid these difficulties by constructing and estimating the retarded sine kernel below the conjugate radius precisely: its direct front has a positive coefficient, so its first-return contribution can be detected without a clean-intersection assumption.

For completeness, we specify the operators and traces used in the
argument.  Choose a real orthonormal eigenbasis
$\{\phi_j\}_{j\ge0}$ in $L^2(M,dV_g)$, with
$\phi_0=\Vol_g(M)^{-1/2}$.  The nonnegative self-adjoint square root of
$\Delta_g$ is defined by
\begin{equation}\label{eq:square-root}
 \begin{aligned}
 \sqrt{\Delta_g}\left(\sum_jc_j\phi_j\right)
     &=\sum_j\sqrt{\lambda_j(g)}\,c_j\phi_j,\\
 \mathcal D(\sqrt{\Delta_g})
     &=\left\{\sum_jc_j\phi_j\in L^2(M):
                         \sum_j\lambda_j(g)|c_j|^2<\infty\right\}.
 \end{aligned}
\end{equation}
This definition is independent of the choice of eigenbasis.  Set
\begin{equation}\label{eq:sine-trace}
 S_g(t)=\frac{\sin(t\sqrt{\Delta_g})}{\sqrt{\Delta_g}},
 \qquad \mathcal S_g=\Tr S_g,
\end{equation}
where $S_g(t)$ acts by multiplication by $t$ on the zero eigenspace.
Its distributional kernel is
\begin{equation}\label{eq:kernel-expansion}
 H(t,x,y)=\frac{t}{\Vol_g(M)}
 +\sum_{j\ge1}\frac{\sin(t\sqrt{\lambda_j(g)})}
                         {\sqrt{\lambda_j(g)}}\phi_j(x)\phi_j(y).
\end{equation}
Thus, for $f\in C^\infty(M)$,
\begin{equation}\label{eq:wave-kernel}
 u(t,x)=S_g(t)f(x)=\int_M H(t,x,y)f(y)\,dV_g(y)
\end{equation}
solves
\begin{equation}\label{eq:wave-equation}
 \begin{cases}
 (\partial_t^2+\Delta_g)u=0,&t>0,\\
 u(0,x)=0,\qquad \partial_tu(0,x)=f(x).
 \end{cases}
\end{equation}
Equivalently, $H(0,x,y)=0$ and
$\partial_tH(0,x,y)=\delta_y(x)$, with the Dirac distribution
normalized relative to $dV_g(x)$.  The spectral expansion also shows
that $H(t,x,y)=H(t,y,x)$ and that the wave equation holds in either
spatial variable.  Only positive times are used in the retarded
construction below.

The trace in \eqref{eq:sine-trace} is a distribution in time.  For
$\psi\in C_c^\infty(\mathbb R)$, set
\begin{equation}\label{eq:spectral-trace-definition}
 \langle\mathcal S_g,\psi\rangle
 =\int_{\mathbb R}t\psi(t)\,dt
 +\sum_{j\ge1}\int_{\mathbb R}\psi(t)
       \frac{\sin(t\sqrt{\lambda_j(g)})}{\sqrt{\lambda_j(g)}}\,dt.
\end{equation}
Repeated integration by parts and Weyl's law give absolute convergence
of this series.  Equivalently, $\langle\mathcal S_g,\psi\rangle$ is the
trace of the smoothing operator $\int\psi(t)S_g(t)\,dt$.  The notation
$\mathcal S_g(t)=\int_M H(t,x,x)\,dV_g(x)$ is understood in this
time-regularized sense; we do not assert that $S_g(t)$ is trace class at
a fixed time.  In particular, $\mathcal S_g$ depends only on the
spectrum.  Within an isospectral class we therefore denote the
common distribution by $\mathcal S_\Lambda$.

The proof has three steps.  First, the heat coefficients
$a_2,a_3,a_4$ give a curvature--length dichotomy.  After normalizing the
maximum curvature, local Sobolev--Morrey estimates show that if
$\ell(g)$ is sufficiently small in terms of these spectral
coefficients, then
\[
                            \ell(g)<\conj(g).
\]
This is the geometric condition needed later: the exponential map is
nonsingular at every return vector with length sufficiently close to
$\ell(g)$.

Second, below the conjugate radius a finite retarded parametrix has a
positive direct front.  Its trace is a positive measure plus a causal
tail and a differentiable remainder.  Near a shortest return, the
cumulative positive mass is at least $cu^{3/2}$ in a time interval of
width $u$, whereas the tail is bounded by $Cu$ times that mass.  If the
trace were $C^1$, subtracting the smooth and differentiable terms would
give an onset integral of order $O(u^2)$, a contradiction.  Hence
$\mathcal S_g$ is not $C^1$ at $\ell(g)$, without any nondegeneracy or
clean-intersection assumption.  A related use of positive principal
amplitudes without clean hypotheses appears for deformations of
Liouville tori in \cite[Theorem~4.13 and Section~4.5]{HKLV2025}.

Third, the same local construction on one fixed realization of
$\Lambda$, inside its injectivity radius, proves that
$\mathcal S_\Lambda$ is $C^1$ on an interval
$(0,\delta_\Lambda)$.  Because this distribution is fixed by the
spectrum, the interval is common to every realization; no uniform
injectivity-radius estimate for the other metrics is used.  If
$\ell(g)<\delta_\Lambda$ and the first step gives
$\ell(g)<\conj(g)$, the second and third steps contradict one another.

Section~\ref{sec:curvature} derives the heat-invariant estimates and the
inequality $\ell<\conj$.  Section~\ref{sec:first-return} constructs the
retarded parametrix and proves the first-return singularity.  Section~
\ref{sec:conclusion} proves the common $C^1$ interval and completes the
proof of Theorem~\ref{thm:main}.

\vspace{2ex}
\noindent\textbf{AI assistance.} This work was developed with substantial assistance from 
 ChatGPT GPT-6 Pro (powered by GPT-6 Astra). In particular, the idea of using the sine propagator \[ S_g(t)=\frac{\sin(t\sqrt{\Delta_g})}{\sqrt{\Delta_g}} \] and the regularity of its trace to derive a contradiction from the existence of sufficiently short closed geodesics was suggested by ChatGPT. ChatGPT also assisted in implementing this idea, including the construction and analysis of the wave-kernel parametrix and the proof of the resulting trace singularity. The author has carefully checked the mathematical arguments, computations in the final manuscript and takes full responsibility for the correctness and content of the paper.

\section{Heat coefficients and the inequality
\texorpdfstring{$\ell<\conj$}{ell < conj}}
\label{sec:curvature}

Set
\begin{equation}\label{eq:heat-convention}
 H_g(t):=\Tr e^{-t\Delta_g}
       \sim(4\pi t)^{-3/2}\sum_{m=0}^{\infty}a_m(g)t^m,
       \qquad t\to 0.
\end{equation}
The conclusion $\ell<\conj$ below will be used in
Section~\ref{sec:first-return}; it ensures that the first return
can be analyzed before the exponential map develops conjugate
points. The present section obtains this condition from spectral
heat data whenever $\ell$ is sufficiently small.

\begin{theorem}\label{thm:curvature}
There is a universal $C_*>0$ such that, if
\begin{equation}\label{eq:Q-def}
 Q(g)=\max\left\{(3C_*a_2(g))^2,
       (3C_*|a_3(g)|)^{2/3},(3C_*|a_4(g)|)^{2/5}\right\}
\end{equation}
and $K(g)=\|\Rm_g\|_{L^\infty}$, then
\begin{equation}\label{eq:dichotomy}
                 \ell(g)\sqrt{K(g)}<1
                 \quad\text{or}\quad K(g)\le Q(g).
\end{equation}
Consequently, $\ell(g)<Q(g)^{-1/2}$ implies
$\ell(g)<\conj(g)$.
\end{theorem}

We will use the following:

\begin{lemma}\label{lem:heat}
Every smooth closed three-manifold satisfies
\begin{align}
 \|\Rm\|_2^2&\le240a_2,\label{eq:heat-Y}\\
 \|\nabla\Rm\|_2^2
  &\le C\left(|a_3|+\int_M|\Rm|^3\dV\right),\label{eq:heat-X}\\
 \|\nabla^2\Rm\|_2^2
  &\le C\left(|a_4|+\int_M|\Rm|^4\dV
            +\int_M|\Rm|\,|\nabla\Rm|^2\dV\right).
                                                    \label{eq:heat-Z}
\end{align}
\end{lemma}
\begin{proof}
The three-dimensional curvature identity
$|\Rm|^2=4|\Ric|^2-R^2$ and the usual $a_2$ formula give
\begin{equation*}
                  360a_2=\int_M(3R^2+6|\Ric|^2)\dV.
\end{equation*}
This proves \eqref{eq:heat-Y}. The leading integrated terms of the
next two coefficients are
\begin{align}
 a_3&=-\frac1{1680}\int_M
       (5|\nabla R|^2+2|\nabla\Ric|^2)\dV+\int_M P_3\dV,
                                                       \label{eq:a3}\\
 a_4&=\frac1{30240}\int_M
       (11|\nabla^2R|^2+2|\nabla^2\Ric|^2)\dV+\int_M P_4\dV.
                                                       \label{eq:a4}
\end{align}
These coefficients follow from the leading-term formula of
Branson, Gilkey and \O rsted \cite{BGO1990}; see
\cite[Theorem 3.3(2), p.~291]{Gilkey2007} for the explicit formulation.
In that convention the coefficient indexed by $2m$ equals
$(4\pi)^{-3/2}a_m$ in \eqref{eq:heat-convention}.

The remaining integrated terms have at least three curvature
factors. At weight six they are cubic without derivatives, so
$|P_3|\le C|\Rm|^3$. At weight eight, the possibilities are four
undifferentiated curvature factors or three factors with two
derivatives in total. Integration by parts in the latter terms gives
\begin{equation*}
 \left|\int_M P_4\dV\right|
 \le C\int_M\bigl(|\Rm|^4+|\Rm|\,|\nabla\Rm|^2\bigr)\dV.
\end{equation*}
Commuting derivatives changes only these same remainder types.
This also explains why no additive constant or volume term occurs
in \eqref{eq:heat-X}--\eqref{eq:heat-Z}. The remainder structure is
the one used in \cite[Theorem 4.2 and p.~303]{BPP1994}.

In dimension three the full curvature tensor is an algebraic linear
expression in $\Ric$, $R$, and the metric. Since the metric is
parallel, the same expression holds after covariant differentiation.
The definite quadratic forms in \eqref{eq:a3}--\eqref{eq:a4}
therefore control the corresponding Riemann-tensor energies.
\end{proof}

\begin{lemma}\label{lem:local-morrey}
For every $i_0>0$ there is $C(i_0)$ such that
\begin{equation}\label{eq:local-morrey}
 \|\Rm\|_\infty\le C(i_0)
   \bigl(\|\Rm\|_2+\|\nabla\Rm\|_2+\|\nabla^2\Rm\|_2\bigr)
\end{equation}
whenever $\|\Rm\|_\infty\le1$ and $\inj\ge i_0$ on a smooth
closed three-manifold. The constant does not depend on total volume
or diameter.
\end{lemma}
\begin{proof}
Fix $s>0$ with $4s<\min\{i_0,1/2\}$. Rauch comparison makes every
normal chart of radius $4s$ uniformly bilipschitz to its Euclidean
ball, since $|\sec|\le1$. A scalar Euclidean Sobolev inequality with
a cutoff and Kato's inequality give, for every smooth tensor $T$,
\begin{equation*}
 \|T\|_{L^6(B(p,2s))}
 \le C(i_0)\bigl(\|T\|_{L^2(B(p,4s))}
                  +\|\nabla T\|_{L^2(B(p,4s))}\bigr).
\end{equation*}
Apply this to $T=\Rm$ and $T=\nabla\Rm$. For the scalar function
$f=|\Rm|$, we have $|df|\le|\nabla\Rm|$ almost everywhere.
Euclidean $W^{1,6}$ embedding on an inner coordinate ball now
bounds $f(p)$ by the right hand side of \eqref{eq:local-morrey}.
Taking the supremum over $p$ proves the result.
\end{proof}

We use the  injectivity-radius identity of Klingenberg,
\begin{equation}\label{eq:klingenberg}
            \inj(g)=\min\{\conj(g),\ell(g)/2\}.
\end{equation}

\begin{proof}[Proof of Theorem~\ref{thm:curvature}]
The case $K=0$ is flat and has infinite conjugate radius. Suppose
$K>0$ and $\ell(g)\sqrt K\ge1$, and set $h=Kg$. Then
\[
 \|\Rm_h\|_\infty=1,\qquad |\sec_h|\le1,
       \qquad \ell(h)\ge1.
\]
Rauch comparison and \eqref{eq:klingenberg} imply
\begin{equation}\label{eq:normalized-inj}
 \conj(h)\ge\pi,\qquad
 \inj(h)\ge\min\{\pi,1/2\}=1/2.
\end{equation}
The normalized heat coefficients are
\begin{equation}\label{eq:scaling}
                       a_m(h)=K^{3/2-m}a_m(g).
\end{equation}
Write $Y=\|\Rm_h\|_2^2$, $X=\|\nabla\Rm_h\|_2^2$, and
$Z=\|\nabla^2\Rm_h\|_2^2$. Since $|\Rm_h|\le1$,
Lemma~\ref{lem:heat} gives
\[
 Y\le240a_2(h),\qquad X\le C(|a_3(h)|+Y),\qquad
 Z\le C(|a_4(h)|+Y+X).
\]
Squaring \eqref{eq:local-morrey}, using
\eqref{eq:normalized-inj}, and increasing a universal constant
$C_*$ if necessary, we obtain
\begin{equation}\label{eq:normalized-energy}
 1\le C_*\bigl(a_2(g)K^{-1/2}
       +|a_3(g)|K^{-3/2}+|a_4(g)|K^{-5/2}\bigr).
\end{equation}
If $K>Q(g)$, each of the three terms on the right, including
$C_*$, is less than $1/3$. This contradiction proves
\eqref{eq:dichotomy}.

If $K=0$, the consequent follows from $\conj(g)=+\infty$.
If $K>0$ and $\ell(g)<Q(g)^{-1/2}$, the alternative
$\ell(g)\sqrt K\ge1$ would give
$K\ge\ell(g)^{-2}>Q(g)$, contradicting \eqref{eq:dichotomy}.
Thus $\ell(g)\sqrt K<1$, and Rauch comparison gives
$\conj(g)\ge\pi/\sqrt K>\ell(g)$.
\end{proof}

\begin{corollary}\label{cor:subquadratic}
Suppose $a_2(g_j)+|a_3(g_j)|+|a_4(g_j)|$ is bounded on a sequence
of smooth closed three-manifolds and $\ell(g_j)\to0$. Then
\[
                \ell(g_j)^2\|\Rm_{g_j}\|_\infty\longrightarrow0.
\]
\end{corollary}
\begin{proof}
If not, pass to a subsequence with
$\ell(g_j)\sqrt{K_j}\ge c>0$, where
$K_j=\|\Rm_{g_j}\|_\infty\to\infty$.
For $h_j=K_jg_j$, replace $1/2$ in
\eqref{eq:normalized-inj} by $\min\{\pi,c/2\}$. The same proof
gives \eqref{eq:normalized-energy} with a constant depending on $c$.
Its right hand side tends to zero, a contradiction.
\end{proof}

\section{Construction of parametrix}
\label{sec:first-return}

We construct a finite retarded parametrix below the conjugate
radius and control its remainder. These estimates will be used in
Theorem~\ref{thm:first-singularity} to prove non-smoothness at the first return.
Every constant in this section may depend on the fixed metric;
no uniform parametrix over the isospectral family is needed.

\subsection{The finite retarded parametrix}\label{subsec:parametrix}

Fix a smooth closed connected Riemannian three-manifold $(M,g)$
and $0<T<r_*<\conj(g)$. We use a finite Hadamard--Riesz
construction; see \cite[Chapter 2]{BGP2007} for the general
framework. The transport equations, local-isometry pushforward,
and remainder estimates needed here are derived below.

\subsubsection{Construction on the exponential pullback ball}

For $x\in M$, the map $\exp_x$ is a local diffeomorphism on
$D_x(r_*)=\{v\in T_xM:|v|<r_*\}$, hence the pullback metric
$\widetilde g_x=\exp_x^*g$ is smooth. We set
\begin{equation}\label{eq:normal-data}
 dV_{\widetilde g_x}=j_x(v)\,dv,
 \qquad U_x(v)=j_x(v)^{-1/2}>0.
\end{equation}
These are smooth functions on the disk bundle $D(r_*)\subset TM$.

Put
$r=|v|$, $b=\partial_r\log j$, and $\sigma=t^2-r^2$.
Let $\widetilde\Delta$ be the nonnegative Laplacian of the pullback
metric.  Suppressing the base point $x$ when convenient, define smooth coefficients by
\begin{align}
 v_0(v)&=-\frac{U(v)}{8\pi}
     \int_0^1\frac{\widetilde\Delta U}{U}(sv)\,ds,
                                                       \label{eq:v0}\\
 v_m(v)&=-\frac{U(v)}{4m}
     \int_0^1s^m\frac{\widetilde\Delta v_{m-1}}{U}(sv)\,ds,
       \qquad m\ge1.                                 \label{eq:vm}
\end{align}
For $t>0$, consider
\begin{equation}\label{eq:PN}
 P_N(t,v)=\frac{U(v)}{4\pi r}\delta(t-r)
                    +\sum_{m=0}^N v_m(v)(t^2-r^2)_+^m,
\end{equation}
where $z_+^0=\1_{[0,\infty)}(z)$. The first term is the spherical
distribution for $t>0$, extended by zero to negative time.

\begin{lemma}\label{lem:transport}
Fix $x\in M$ and write $r=|v|_{g_x}$. The distribution $P_N$
defined in \eqref{eq:PN} satisfies the Cauchy problem
\begin{equation}\label{eq:PN-error}
\begin{cases}
(\partial_t^2+\widetilde\Delta)P_N
   =(\widetilde\Delta v_N)(t^2-r^2)_+^N,
   & \text{in }(0,\infty)\times D_x(r_*),\\[1mm]
P_N|_{t=0}=0,\qquad
\partial_tP_N|_{t=0}=\delta_0.
\end{cases}
\end{equation}
Here $\widetilde\Delta$ acts in the spatial variable $v$, and
$\delta_0$ is the Dirac distribution at $v=0$, normalized by
\[
 \int_{D_x(r_*)}\delta_0(v)\eta(v)\,
          dV_{\widetilde g_x}(v)=\eta(0)
 \qquad
 \text{for every }\eta\in C_c^\infty(D_x(r_*)).
\]
More precisely, for every
\[
 \varphi\in C_c^\infty([0,\infty)\times D_x(r_*)),
\]
we have
\[
\begin{aligned}
 &\int_0^\infty\int_{D_x(r_*)}
 P_N(t,v)(\partial_t^2+\widetilde\Delta)\varphi(t,v)
 \,dV_{\widetilde g_x}(v)\,dt\\
 &\qquad=
 \int_0^\infty\int_{D_x(r_*)}
 (\widetilde\Delta v_N)(v)(t^2-r^2)_+^N
 \varphi(t,v)\,dV_{\widetilde g_x}(v)\,dt
 +\varphi(0,0).
\end{aligned}
\]
\end{lemma}
\begin{proof}
We first compute away from the cone vertex $(t,v)=(0,0)$.
For a smooth spatial function $a$ and a one-variable distribution
$F$, the radial formula for the nonnegative Laplacian gives
\begin{equation}\label{eq:radial-identity}
 \begin{aligned}
 (\partial_t^2+\widetilde\Delta)(aF(\sigma))
 ={}&4\sigma a F''(\sigma)\\
 &+(8a+2rba+4r\partial_ra)F'(\sigma)
       +(\widetilde\Delta a)F(\sigma).
 \end{aligned}
\end{equation}
Indeed, $\partial_t\sigma=2t$, $\nabla\sigma=-2r\partial_r$,
and $\widetilde\Delta\sigma=6+2rb$.
The distributional compositions in \eqref{eq:radial-identity}
are well defined away from the vertex, since $d\sigma\ne0$
there on the cone.

For $t>0$, denote the direct term by
\[
 D(t,v)=\frac{U(v)}{4\pi r}\delta(t-r)
       =\frac{U(v)}{2\pi}\delta(\sigma).
\]
The identities $\sigma\delta''=-2\delta'$ and
$2\partial_rU+bU=0$ cancel its $\delta'$ coefficient, so
\[
 (\partial_t^2+\widetilde\Delta)D
       =\frac{\widetilde\Delta U}{2\pi}\delta(\sigma).
\]
For the first tail, $F=\1_{[0,\infty)}$, the identity
$\sigma\delta'=-\delta$ yields
\[
 (\partial_t^2+\widetilde\Delta)(v_0\sigma_+^0)
 =4U\bigl(r\partial_r(v_0/U)+v_0/U\bigr)\delta(\sigma)
       +(\widetilde\Delta v_0)\sigma_+^0.
\]
Thus the delta term is canceled precisely when
\[
 r\partial_r(v_0/U)+v_0/U
             =-\frac{\widetilde\Delta U}{8\pi U}.
\]
The unique solution smooth at $v=0$ is \eqref{eq:v0}.

For $m\ge1$, \eqref{eq:radial-identity} gives
\[
 \begin{aligned}
 (\partial_t^2+\widetilde\Delta)(v_m\sigma_+^m)
 ={}&4mU\bigl(r\partial_r(v_m/U)+(m+1)v_m/U\bigr)
                                      \sigma_+^{m-1}\\
 &+(\widetilde\Delta v_m)\sigma_+^m.
 \end{aligned}
\]
For $m=1$, the second derivative of $\sigma_+$ contributes a
delta distribution, but its factor $\sigma$ makes that term zero.
Cancellation with the preceding residual is therefore equivalent to
\[
 r\partial_r(v_m/U)+(m+1)v_m/U
             =-\frac{\widetilde\Delta v_{m-1}}{4mU}.
\]
Solving along each radial segment gives \eqref{eq:vm}.
These ray integrals are smooth jointly in the base point and
$v$, including $v=0$. Successive cancellation leaves exactly
$(\widetilde\Delta v_N)\sigma_+^N$.

It remains to check the initial data at the vertex. For a smooth
spatial test function $\phi$, the direct term has pairing
\begin{equation}\label{eq:vertex}
 \langle D(t,\cdot),\phi\rangle
 =\frac{t}{4\pi}\int_{\mathbb S^2}
       \sqrt{j(t\omega)}\,\phi(t\omega)\,d\omega
 =t\phi(0)+O(t^3).
\end{equation}
Here $j(v)=1+O(|v|^2)$ in normal coordinates, and the linear
part of $\phi(t\omega)$ has zero spherical average. After
$v=tw$, the pairing of the $m$th tail is $t^{2m+3}$ times a
smooth function of $t\ge0$. Hence every tail has zero initial
displacement and velocity, whereas \eqref{eq:vertex} gives
initial velocity $\delta_0$ for the direct term. The causal
extension therefore has the additional source
$\delta(t)\delta_0$ and no derivative-of-delta source at $t=0$.
This proves the stated Cauchy and weak formulations.
\end{proof}

\subsubsection{Pushforward and the remainder}

Choose $\chi\in C_c^\infty([0,r_*))$ equal to one on $[0,T]$
and constant near zero. Multiply \eqref{eq:PN} by $\chi(r)$.
For $0\le t<T$, every term is supported in $r\le t$,
so derivatives of $\chi$ produce no error on this time interval.
For the $C^1$ estimate below, take $N=12$. Higher finite orders
will be used in the remark at the end of Section~\ref{sec:conclusion}.
The causal residual
\begin{equation}\label{eq:smooth-error}
 \chi(r)(\widetilde\Delta v_N)(x,v)\,
                 \1_{\{t\ge0\}}(t^2-r^2)_+^N
\end{equation}
is $C^{N-1}$ jointly in $(t,x,v)$, including the cone vertex.
Away from the vertex this follows from the regularity of
$z\mapsto z_+^N$; at the vertex, derivatives of order at most
$N-1$ of the positive-time expression vanish as $(t,v)\to(0,0)$.

Let
\[
 \begin{gathered}
 D(r_*)=\{(x,v)\in TM:|v|_{g_x}<r_*\},\\
 \Phi:D(r_*)\longrightarrow M\times M,
 \qquad \Phi(x,v)=(x,\exp_xv).
 \end{gathered}
\]
This is a local diffeomorphism. The support imposed by $\chi$
is compact in the disk bundle and is covered by finitely many
local inverse charts. Push the parametrix and its residual
forward fiberwise, using $dV_{\widetilde g_x}$ on the fiber.
The resulting kernels $K_N(t,x,y)$ and $E_N(t,x,y)$ are
specified, distributionally in $t$, by
\begin{equation}\label{eq:pushforward-Q}
 \begin{aligned}
 \int_M K_N(t,x,y)\varphi(y)\,dV_g(y)
  =\int_{D_x(r_*)}\chi(|v|)P_N(t,x,v)
              \varphi(\exp_xv)\,dV_{\widetilde g_x}(v)
 \end{aligned}
\end{equation}
and
\begin{equation}\label{eq:pushforward-E}
 \begin{aligned}
 &\int_M E_N(t,x,y)\varphi(y)\,dV_g(y)\\
 &\quad=\int_{D_x(r_*)}\chi(|v|)(\widetilde\Delta v_N)(x,v)
       \1_{\{t\ge0\}}(t^2-r^2)_+^N
              \varphi(\exp_xv)\,dV_{\widetilde g_x}(v),
 \end{aligned}
\end{equation}
for every $\varphi\in C^\infty(M)$.
The exponential map is a local isometry from its pullback metric,
so these pushforwards are sums over local inverse branches with
no further Jacobian factor relative to $dV_g(y)$. Compactly
supported integration by parts shows that pushforward commutes
with the Laplacian in the $y$ variable. Consequently
\begin{equation}\label{eq:pushed-data}
 \begin{cases}
 (\partial_t^2+\Delta_y)K_N(t,x,y)=E_N(t,x,y),&0<t<T,\\
 K_N(0,x,y)=0,\qquad \partial_tK_N(0,x,y)=\delta_x(y).
 \end{cases}
\end{equation}
The compactly supported local pushforward also preserves the
joint $C^{N-1}$ regularity of the residual.

Use the usual operator-kernel convention
\[
 Q_N(t)f(x)=\int_M K_N(t,x,y)f(y)\,dV_g(y),\qquad
 T_N(t)f(x)=\int_M E_N(t,x,y)f(y)\,dV_g(y).
\]
Equation \eqref{eq:pushed-data} then reads
$Q_N''+Q_N\Delta_g=T_N$ on smooth functions, with
$Q_N(0)=0$ and $Q_N'(0)=I$. The order of composition matters:
the pushed kernel has not been assumed to be symmetric.
Define the remainder operator and its kernel by
\[
 \mathcal R_N(t)=S_g(t)-Q_N(t),\qquad
 R_N(t,x,y)=H(t,x,y)-K_N(t,x,y).
\]

\begin{proposition}\label{prop:remainder}
For $0<t<T$, the kernel $R_N$ is jointly $C^1$ in $(t,x,y)$.
The operators $\mathcal R_N(t)$ and $\partial_t\mathcal R_N(t)$
are trace class and continuous in trace norm. Moreover,
\begin{equation}\label{eq:trace-remainder}
 \Tr\mathcal R_N(t)=\int_M R_N(t,x,x)\,dV_g(x)
\end{equation}
is a $C^1$ function of $t$.
\end{proposition}
\begin{proof}
The exact wave kernel satisfies the wave equation in $y$.
Subtracting \eqref{eq:pushed-data} gives
\[
 \begin{cases}
 (\partial_t^2+\Delta_y)R_N(t,x,y)=-E_N(t,x,y),&0<t<T,\\
 R_N(0,x,y)=0,\qquad \partial_tR_N(0,x,y)=0.
 \end{cases}
\]
Uniqueness for the distributional Cauchy problem and Duhamel's
formula therefore yield
\begin{equation}\label{eq:RN-kernel}
 R_N(t,x,y)=-\int_0^t\int_M E_N(s,x,z)H(t-s,z,y)
                                      \,dV_g(z)\,ds.
\end{equation}
Equivalently, $\mathcal R_N(t)=-\int_0^t T_N(s)S_g(t-s)\,ds$.

Work in $H^8(M\times M)$, defined using
$1+\Delta_x+\Delta_y$. The propagator acting in the $y$
variable and its first time derivative commute with this
operator and are bounded on a fixed time interval. Since
$E_N\in C_tH^8$, \eqref{eq:RN-kernel} implies
\[
                  R_N\in C_t^1H^8(M\times M).
\]
Sobolev embedding on the six-dimensional product gives joint
$C^1$ regularity. These statements hold up to $t=0$ on each
compact subinterval of $[0,T)$.

For the trace-class assertion, factor
\[
 \mathcal R_N(t)=B_N(t)(1+\Delta_g)^{-2},
\]
where $B_N(t)$ has kernel $(1+\Delta_y)^2R_N(t,x,y)$.
This kernel belongs to $L^2(M\times M)$, continuously with
one time derivative, so $B_N$ is $C^1$ with values in the
Hilbert--Schmidt operators. Weyl's law in dimension three
implies that $(1+\Delta_g)^{-2}$ is Hilbert--Schmidt.
Their product is therefore $C^1$ in trace norm.

Finally, apply heat smoothing in both variables. It approximates
the continuous kernel uniformly and the corresponding operator
in trace norm. The diagonal trace identity for the smoothed
kernels passes to the limit and proves \eqref{eq:trace-remainder}.
The same argument applies to the time derivative.
\end{proof}

\begin{remark}\label{rem:time-smearing}
All subsequent trace formulas can be justified without restricting
an unsmoothed distributional kernel to the diagonal. For
$\psi\in C_c^\infty((0,T))$, the operator
$\int\psi(t)S_g(t)\,dt$ is smoothing by repeated integration by
parts in its spectral expansion. In the direct part of $K_N$,
time integration replaces $\delta(t-r)/(4\pi r)$ by
$\psi(r)/(4\pi r)$, which vanishes near $r=0$. The integrated
tails are smooth for $r>0$; near $r=0$ their radial factors
are polynomials in $r^2$, since $\psi$ vanishes near zero.
The time-smoothed $Q_N$ thus also has a smooth kernel, and
its trace is its diagonal integral. Together with
Proposition~\ref{prop:remainder}, this proves the trace
identities below by duality against $\psi$.
\end{remark}

\section{The proof of the Main Theorem}\label{sec:conclusion}
We first establish the first-return singularity using the preceding
parametrix. We then construct a common interval of regularity for the
spectral trace and combine the two conclusions with
Theorem~\ref{thm:curvature}.

\begin{theorem}\label{thm:first-singularity}
If $\ell(g)<\conj(g)$, the distribution $\mathcal S_g$ is not
represented by a $C^1$ function on any neighborhood of $\ell(g)$.
In particular, it is not smooth there.
\end{theorem}

\begin{proof}
Write $\ell=\ell(g)$ and choose
\[
             \ell<T_1<T<r_*<\conj(g).
\]
Throughout the proof, $N=12$, and $K_N$, $Q_N$, and
$\mathcal R_N$ are the parametrix kernel, its associated operator,
and the remainder from Section~\ref{subsec:parametrix}.

We first verify that every nonzero return vector has length at
least $\ell$. By \eqref{eq:klingenberg} and the hypothesis
$\ell<\conj(g)$, we have $\inj(g)=\ell/2$. Suppose that
$\exp_xv=x$ with $0<L:=|v|<\ell$, and let
$\gamma:[0,L]\to M$ be the associated unit-speed geodesic.
The two curves
\[
 t\longmapsto\gamma(t),\qquad
 t\longmapsto\gamma(L-t),\qquad 0\le t\le L/2,
\]
are minimizing geodesics from $x$ to $\gamma(L/2)$, each of
length less than $\inj(g)$. Uniqueness of a minimizing geodesic
at this distance makes them identical. At $t=L/2$ this would
give $\dot\gamma(L/2)=-\dot\gamma(L/2)$, a contradiction.
Consequently,
\begin{equation}\label{eq:no-short-return}
       \exp_xv=x,\quad v\ne0\quad\Longrightarrow\quad |v|\ge\ell.
\end{equation}
This argument concerns all geodesic loops.

Consider the nonzero return set
\[
 \mathscr R=\{(x,v)\in D(r_*):v\ne0,\ \exp_xv=x\}.
\]
The differential of $v\mapsto\exp_xv$ is invertible on
$D(r_*)$. The implicit function theorem therefore makes
$\mathscr R$ a smooth manifold, with its projection to $M$ a
local diffeomorphism. In particular, its local branches have
the form $x\mapsto v_\alpha(x)$, where
\[
 \exp_xv_\alpha(x)=x,\qquad
 L_\alpha(x):=|v_\alpha(x)|>0.
\]
The subset
\[
 \mathscr R_{[\ell,T_1]}
   =\{(x,v)\in\mathscr R:\ell\le |v|\le T_1\}
\]
is compact. Choose finitely many branch charts covering this
set and nonnegative smooth weights $\rho_\alpha$, compactly
supported in these charts, whose sum is one on a neighborhood
of $\mathscr R_{[\ell,T_1]}$ in $\mathscr R$. We also write
$\rho_\alpha(x)$ for each weight in its branch coordinates,
and denote the corresponding open subset of $M$ by
$\Omega_\alpha$. These weights count each return vector once,
even if several charts cover it; distinct return vectors at the
same point are all retained.

Choose an open interval $J$ containing $\ell$ with
$\overline J\subset(0,T_1)$, and let
$\psi\in C_c^\infty(J)$, extended by zero outside $J$.
We compute the time-regularized trace before restricting a
kernel to the diagonal. The pushforward formula and
\eqref{eq:PN} give
\begin{equation}\label{eq:regularized-parametrix}
\begin{aligned}
 \int_0^T\psi(t)K_N(t,x,y)\,dt
  =\sum_{\substack{\exp_xv=y\\r=|v|<T}}
  \left[
     \frac{U_x(v)}{4\pi r}\psi(r)
     +\sum_{m=0}^N v_m(x,v)
        \int_r^T\psi(t)(t^2-r^2)^m\,dt
  \right].
\end{aligned}
\end{equation}
For the direct term, the summand at $v=0$ is defined to be
zero: $\psi(r)/r$ vanishes identically near $r=0$.
Terms with $r\ge T$ vanish by causal support, and the cutoff
$\chi$ equals one on the remaining support. The sum is finite
locally, since the exponential map is a local diffeomorphism
and only a compact subset of its disk bundle contributes.

The kernel in \eqref{eq:regularized-parametrix} is smooth.
Indeed, the direct term is smooth after time integration, and
for $r>0$ the integrated tails are smooth functions of $r$.
Near $r=0$, the support of $\psi$ stays away from zero, so
\[
       \int_r^T\psi(t)(t^2-r^2)^m\,dt
\]
is a polynomial in $r^2$. Thus the ordinary diagonal trace
of \eqref{eq:regularized-parametrix} is legitimate, as in
Remark~\ref{rem:time-smearing}.

Set
\begin{equation}\label{eq:first-return-amplitude}
 a_\alpha(x)=\rho_\alpha(x)
            \frac{U_x(v_\alpha(x))}{4\pi L_\alpha(x)}\ge0,
 \qquad
 \mu=\sum_\alpha(L_\alpha)_*(a_\alpha\,dV_g).
\end{equation}
The measure $\mu$ is finite and positive. Each measure in
the sum is defined on $\Omega_\alpha$. By
\eqref{eq:no-short-return}, $\mu$ is supported in
$[\ell,\infty)$. In particular, its restriction to $J$ has
no contribution to the left of $\ell$. The direct term in
the time-regularized trace is exactly
\begin{equation}\label{eq:positive-measure}
 \sum_\alpha\int_{\Omega_\alpha}
       a_\alpha(x)\psi(L_\alpha(x))\,dV_g(x)
                  =\int_J\psi(t)\,d\mu(t).
\end{equation}

We next prove a lower bound for its cumulative mass. Choose
a shortest  closed geodesic, a point $x_0$ on it,
and the corresponding return vector $v_*$ of length $\ell$.
At least one of the partition weights is positive at
$(x_0,v_*)$. Choose such a chart, denoted by $\alpha_0$.
After shrinking its base neighborhood, we have
$a_{\alpha_0}\ge a_*>0$. On this neighborhood,
\[
 L_{\alpha_0}(x)\ge\ell=L_{\alpha_0}(x_0),
 \qquad dL_{\alpha_0}(x_0)=0.
\]
A bound for the Hessian of the smooth function
$L_{\alpha_0}$ gives constants $C\ge1$ and $\rho>0$ such
that
\[
 0\le L_{\alpha_0}(x)-\ell
           \le C\,d_g(x,x_0)^2,
       \qquad x\in B_g(x_0,\rho).
\]
For sufficiently small $u>0$, the ball
$B_g(x_0,\sqrt{u/(2C)})$ is therefore contained in
$\{\ell\le L_{\alpha_0}<\ell+u\}$. The local volume
comparison for this one fixed smooth metric yields
\begin{equation}\label{eq:first-return-mass}
 \mu([\ell,\ell+u])
 \ge a_*\Vol_g\bigl(B_g(x_0,\sqrt{u/(2C)})\bigr)
 \ge c u^{3/2}.
\end{equation}
No lower bound for the Hessian, and hence no nondegeneracy
or clean-intersection hypothesis, is used here.

We must keep the tail terms in the trace. The return vector
$v=0$ contributes the smooth function
\begin{equation}\label{eq:zero-return-tail}
          A_N(t)=\sum_{m=0}^N t^{2m}
                         \int_M v_m(x,0)\,dV_g(x).
\end{equation}
For the nonzero return vectors define, for $t\in J$,
\begin{equation}\label{eq:nonzero-return-tail}
 B(t)=\sum_\alpha\int_{\Omega_\alpha}
   \rho_\alpha(x)\sum_{m=0}^N
      v_m(x,v_\alpha(x))
      (t^2-L_\alpha(x)^2)^m
      \1_{\{L_\alpha(x)\le t\}}\,dV_g(x).
\end{equation}
Values of the indicator at its boundary do not affect the
distribution represented by this function. All relevant
nonzero return vectors have length in $[\ell,T_1]$ and
are covered by the chosen partition. On the compact supports
of the weights, $U_x(v)/(4\pi|v|)$ is strictly positive,
and every coefficient $v_m(x,v)$ is bounded. Since $t$ stays
in the bounded interval $J$, there is a constant $C_B$ such
that
\begin{equation}\label{eq:causal-tail-bound}
 \begin{aligned}
 B(t)&=0 &&(t<\ell),\\
 |B(t)|&\le C_B\mu([\ell,t]) &&(t\ge\ell,\ t\in J).
 \end{aligned}
\end{equation}
In particular, $B$ is a locally bounded measurable
function. This bound allows either sign for the tail.

Taking the trace in \eqref{eq:regularized-parametrix} and
using Proposition~\ref{prop:remainder}, we obtain the
distributional identity on $J$
\begin{equation}\label{eq:first-return-decomposition}
       \mathcal S_g
          =\mu+A_N(t)+B(t)+\Tr\mathcal R_N(t).
\end{equation}
More explicitly, its pairing with $\psi$ is
\[
 \langle\mathcal S_g,\psi\rangle
   =\int_J\psi\,d\mu
     +\int_J\psi(t)
       \bigl(A_N(t)+B(t)+\Tr\mathcal R_N(t)\bigr)\,dt.
\]

Suppose, for contradiction, that $\mathcal S_g$ is represented
by a $C^1$ function near $\ell$. Shrink $J$ if necessary and
define
\[
       F(t)=\mathcal S_g(t)-A_N(t)-\Tr\mathcal R_N(t).
\]
Then $F\in C^1(J)$ and
\begin{equation}\label{eq:C1-density-identity}
                    F(t)\,dt=d\mu(t)+B(t)\,dt.
\end{equation}
Since $B$ is locally bounded, this identity implies that
$\mu$ has a locally integrable density on $J$; in particular,
it has no atoms there. By \eqref{eq:no-short-return} and
\eqref{eq:causal-tail-bound}, $F$ vanishes on
$J\cap(-\infty,\ell)$. Thus $F(\ell)=F'(\ell)=0$, and
the $C^1$ bound gives
\begin{equation}\label{eq:C1-onset}
       \left|\int_\ell^{\ell+u}F(t)\,dt\right|
                          \le C_Fu^2
\end{equation}
for all sufficiently small $u>0$.
On the other hand, \eqref{eq:C1-density-identity} and
\eqref{eq:causal-tail-bound} give
\begin{align*}
 \int_\ell^{\ell+u}F(t)\,dt
  &=\mu([\ell,\ell+u])+\int_\ell^{\ell+u}B(t)\,dt\\
  &\ge\mu([\ell,\ell+u])
           -C_B\int_\ell^{\ell+u}\mu([\ell,t])\,dt\\
  &\ge(1-C_Bu)\mu([\ell,\ell+u])
   \ge\frac{c}{2}u^{3/2}
\end{align*}
after decreasing $u$ once more. This contradicts
\eqref{eq:C1-onset} as $u\to 0$ and proves the theorem.
\end{proof}

\begin{lemma}\label{lem:time-gap}
For each spectrum $\Lambda$ realized by a smooth closed connected
three-manifold, there is $\delta_\Lambda>0$ such that the common sine trace
of all realizations of $\Lambda$ is represented by a $C^1$ function on
$(0,\delta_\Lambda)$.
\end{lemma}
\begin{proof}
Fix one realization $(M_0,g_0)$ of $\Lambda$ and choose
\[
 0<T<r_*<\inj(g_0)\le\conj(g_0).
\]
Apply the construction of Section~\ref{sec:first-return} to this metric,
with $N=12$. Since $\exp_x$ is injective on $D_x(r_*)$, the only
return vector in this ball is $v=0$. For a test function
$\psi\in C_c^\infty((0,T))$, the time-integrated direct term vanishes
near $v=0$, because $\psi(r)/r=0$ for all sufficiently small $r$.
The diagonal contribution of the tails is therefore exactly
\[
 A_N(t)=\sum_{m=0}^N t^{2m}
                    \int_{M_0}v_m(x,0)\,dV_{g_0}(x),\qquad 0<t<T.
\]
The time-regularized trace identity consequently gives
\[
 \mathcal S_{g_0}=A_N+\Tr\mathcal R_N
                 \quad\hbox{in }\mathcal D'((0,T)).
\]
Here $A_N$ is a polynomial and $\Tr\mathcal R_N$ is $C^1$ by
Proposition~\ref{prop:remainder}. Hence $\mathcal S_{g_0}$ is represented
by a $C^1$ function on $(0,T)$.

For every other realization $(M,g)$ of $\Lambda$,
\eqref{eq:spectral-trace-definition} gives
$\mathcal S_g=\mathcal S_{g_0}$ as distributions. Thus the same function
and the same interval work for every realization, and we may set
$\delta_\Lambda=T$. Only the fixed metric $g_0$ was used to construct
this interval.
\end{proof}

\begin{proof}[Proof of Theorem~\ref{thm:main}]
The heat coefficients $a_2,a_3,a_4$ are determined by $\Lambda$,
so the quantity $Q(g)$ in \eqref{eq:Q-def} has a common value
$Q_\Lambda$ on the isospectral class. Set
\[
 \varepsilon_\Lambda
    =\min\{\delta_\Lambda,Q_\Lambda^{-1/2}\}>0,
\]
where $0^{-1/2}=+\infty$ and $\delta_\Lambda$ is supplied by
Lemma~\ref{lem:time-gap}. Suppose that a realization $(M,g)$ of
$\Lambda$ satisfied $\ell(g)<\varepsilon_\Lambda$.
Theorem~\ref{thm:curvature} would give $\ell(g)<\conj(g)$, and
Theorem~\ref{thm:first-singularity} would then imply that
$\mathcal S_g$ is not $C^1$ on any neighborhood of $\ell(g)$.
But $0<\ell(g)<\delta_\Lambda$, so such a neighborhood lies
inside the common $C^1$ interval from Lemma~\ref{lem:time-gap}.
This contradiction proves $\ell(g)\ge\varepsilon_\Lambda$ for
every realization of $\Lambda$.

Anderson's compactness criterion, in the formulation recalled in
\cite[Theorem~3.5]{Perry2003}, now gives smooth compactness modulo
diffeomorphisms and finitely many diffeomorphism types in the
isospectral class.
\end{proof}

\providecommand{\bysame}{\leavevmode\hbox to3em{\hrulefill}\thinspace}
\providecommand{\MR}{\relax\ifhmode\unskip\space\fi MR }
\providecommand{\MRhref}[2]{%
  \href{http://www.ams.org/mathscinet-getitem?mr=#1}{#2}}
\providecommand{\href}[2]{#2}

\end{document}